\pdfoutput=1
\RequirePackage{ifpdf}
\ifpdf % We are running pdfTeX in pdf mode
\documentclass[pdftex]{sigma}
\else
\documentclass{sigma}
\fi

\numberwithin{equation}{section}

\newtheorem{Theorem}{Theorem}[section]
\newtheorem{Proposition}[Theorem]{Proposition}
 { \theoremstyle{definition}
\newtheorem{Definition}[Theorem]{Definition}
\newtheorem{Remark}[Theorem]{Remark} }

\begin{document}

\allowdisplaybreaks

\newcommand{\arXivNumber}{1907.11421}

\renewcommand{\thefootnote}{}

\renewcommand{\PaperNumber}{051}

\FirstPageHeading

\ShortArticleName{Dual Invertible Polynomials with Permutation Symmetries}

\ArticleName{Dual Invertible Polynomials with Permutation\\ Symmetries and the Orbifold Euler Characteristic\footnote{This paper is a~contribution to the Special Issue on Algebra, Topology, and Dynamics in Interaction in honor of Dmitry Fuchs. The full collection is available at \href{https://www.emis.de/journals/SIGMA/Fuchs.html}{https://www.emis.de/journals/SIGMA/Fuchs.html}}}

\Author{Wolfgang EBELING~$^\dag$ and Sabir M. GUSEIN-ZADE~$^\ddag$}

\AuthorNameForHeading{W.~Ebeling and S.M.~Gusein-Zade}

\Address{$^\dag$~Leibniz Universit\"{a}t Hannover, Institut f\"{u}r Algebraische Geometrie,\\
\hphantom{$^\dag$}~Postfach 6009, D-30060 Hannover, Germany}
\EmailD{\href{mailto:ebeling@math.uni-hannover.de}{ebeling@math.uni-hannover.de}}

\Address{$^\ddag$~Moscow State University, Faculty of Mechanics and Mathematics,\\
\hphantom{$^\ddag$}~Moscow, GSP-1, 119991, Russia}
\EmailD{\href{mailto:sabir@mccme.ru}{sabir@mccme.ru}}

\ArticleDates{Received July 29, 2019, in final form June 01, 2020; Published online June 11, 2020}

\Abstract{P.~Berglund, T.~H\"ubsch, and M.~Henningson proposed a method to construct mirror symmetric Calabi--Yau manifolds. They considered a pair consisting of an invertible polynomial and of a finite (abelian) group of its diagonal symmetries together with a dual pair. A.~Takahashi suggested a method to generalize this construction to symmetry groups generated by some diagonal symmetries and some permutations of variables. In a previous paper, we explained that this construction should work only under a special condition on the permutation group called parity condition (PC). Here we prove that, if the permutation group is cyclic and satisfies PC, then the reduced orbifold Euler characteristics of the Milnor fibres of dual pairs coincide up to sign.}


\Classification{14J33; 57R18; 32S55}

\begin{flushright}
\begin{minipage}{65mm}
\it Dedicated to Dmitry Borisovich Fuchs\\
on the occasion of his 80th birthday
\end{minipage}
\end{flushright}

\renewcommand{\thefootnote}{\arabic{footnote}}
\setcounter{footnote}{0}

\section{Introduction} \label{sect:Intro}

The idea of mirror symmetry came to mathematics from physics. In the simplest form, it refers to
the observation that there exist pairs of Calabi--Yau manifolds with symmetric sets of Hodge numbers.
It implies, in particular, that their Euler characteristics coincide up to sign.
In \cite{BH2, BH1}, P.~Berglund, T.~H\"ubsch, and M.~Henningson suggested a method to construct mirror symmetric Calabi--Yau manifolds. They considered pairs $(f,G)$ consisting of a quasihomogeneous polynomial~$f$ of a special type (an invertible one) and of a finite (abelian) group $G$ of its diagonal symmetries. For a~pair $(f,G)$ they constructed a dual pair $\big(\widetilde{f}, \widetilde{G}\big)$. For certain pairs $(f,G)$, a~crepant resolution of the quotient $\{f =0 \}/G$ of the subvariety defined by the equation $f=0$ in the weighted projective space is a Calabi--Yau manifold. Berglund, H\"ubsch, and Henningson claimed that the manifolds constructed for the pairs $(f,G)$ and $\big(\widetilde{f}, \widetilde{G}\big)$ are mirror symmetric to each other. Berglund and Henningson~\cite{BH2} proved a symmetry property for the elliptic genera of them (see also~\cite{Kawai-Yang}).

Instead of working with the hypersurface $\{ f=0 \}$ in the weighted projective space one can consider
the Milnor fibre $V_f=\{ f=1 \}$ in the affine space with the action of the group~$G$.
In this case one has to compare orbifold Hodge numbers of the Milnor fibres of dual pairs
and thus the reduced orbifold Euler characteristics of them. There were some
symmetries found for invariants of the pairs $(V_f, G)$ and $\big(V_{\widetilde{f}}, \widetilde{G}\big)$.
In particular, in~\cite{EG-MMJ}, it was shown that the reduced orbifold Euler characteristics
$\overline{\chi}(V_f,G)$ and $\overline{\chi}\big(V_{\widetilde{f}}, \widetilde{G}\big)$ coincide up to sign. (This statement holds for arbitrary pairs $(f,G)$, not only for those giving Calabi--Yau manifolds.)
Besides that, in~\cite{EG-BLMS}, another special sort of symmetry (called Saito duality) was found between the reduced equivariant Euler characteristics $\overline{\chi}^G(V_f)$ and
$\overline{\chi}^{\widetilde{G}}\big(V_{\widetilde{f}}\big)$ (with values in the Burnside rings of the groups) of the Milnor fibres. Initially one could not see a relation of this symmetry with the
mirror one. Later it was understood that the statement for the orbifold Euler characteristics can be
deduced from the Saito duality (in the case of abelian (!) groups; see a discussion below, and see~\cite{EG-1811} for a proof of this fact).

Based on an idea of A.~Takahashi, in \cite{EG-IMRN}, the notion of dual pair was generalized to the following situation. Let $f$ be an invertible polynomial in $n$ variables, let $S$ be a subgroup of the group $S_n$ of permutations of the variables preserving the polynomial $f$, and let $G$ be a group of diagonal symmetries of $f$ invariant with respect to $S$. In this case, the semidirect product $G \rtimes S$ is defined and $f$ is $G \rtimes S$-invariant. (The group $G \rtimes S$ is, in general, not abelian.) One can see that the polynomial $\widetilde{f}$ participating in the BHH-dual pair $\big(\widetilde{f}, \widetilde{G}\big)$ is preserved by the group $S$ and that the dual subgroup $\widetilde{G}$ is $S$-invariant. Therefore, $\widetilde{f}$ is invariant with respect to the semidirect product $\widetilde{G} \rtimes S$. The Berglund--H\"ubsch--Henningson--Takahashi (BHHT) dual to the pair $(f, G \rtimes S)$ is the pair $\big(\widetilde{f},\widetilde{G} \rtimes S\big)$.

In \cite{EG-IMRN}, a special property of a subgroup $S$ of the permutation group $S_n$ was introduced which was called parity condition (PC). It was shown that a non-abelian analogue of the Saito duality between the reduced equivariant Euler characteristics of the Milnor fibres may hold for BHHT-dual pairs only if the permutation group $S$ satisfies PC. This led to the conjecture that BHHT-dual pairs correspond to mirror symmetric varieties only if the condition PC is satisfied. This conjecture found a support in data about Calabi--Yau threefolds presented in~\cite{Yu}.

One invariant which has to be the same up to sign for mirror symmetric orbifolds is the reduced orbifold Euler characteristic. One can conjecture that the reduced orbifold Euler characteristics of BHHT-dual pairs satisfying the PC condition coincide up to sign. In~\cite{EG-1811}, this conjecture was proved for a very particular case, namely when the polynomial $f$ is atomic of loop type (see the definition in Section~\ref{sect:orbifold}).

Here we prove the conjecture for BHHT-dual pairs with a cyclic permutation group, i.e., $S$~is a cyclic group.

\section{Invertible polynomials and non-abelian duality} \label{sect:invertible_poly}
A polynomial $f$ in $n$ variables is called {\em invertible} if it is quasihomogeneous, consists of $n$ monomials, that is
\[ f(x_1, \dots , x_n)=\sum_{i=1}^n a_i \prod_{j=1}^n x_j^{E_{ij}} ,
\]
where $a_i$ are non-zero complex numbers and the matrix $E=(E_{ij})$ has non-negative integer entries, and
$\det E \neq 0$. This does not imply that $f$ has an isolated critical point at the origin, e.g., $f(x_1,x_2)=x_1^4+x_1^2x_2^2$ is an invertible polynomial with a non-isolated critical point at the origin. If $f$ has an isolated critical point at the origin, then the invertible polynomial is called {\em non-degenerate}. Here we will only consider non-degenerate invertible polynomials and we drop the adjective non-degenerate. Without loss of generality one may assume
that $a_i=1$ for $i=1, \dots , n$.

The {\em group of $($diagonal$)$ symmetries} of $f$ is
\[
G_f=\big\{\underline{\lambda}=(\lambda_1, \dots, \lambda_n)\in({\mathbb C}^*)^n\colon
f(\lambda_1x_1, \dots, \lambda_nx_n)=f(x_1, \dots, x_n)\big\} .
\]
One can see that $G_f$ is an abelian group of order $\vert G_f \vert = \vert\det E \vert$.

The group $S_n$ of permutations on $n$ elements acts on ${\mathbb C}^n$ by permuting the coordinates. Suppose that the polynomial $f$ is invariant with respect to the action of a subgroup $S$ of $S_n$. In this case, $S$ acts on the group $G_f$ by conjugation. The group of transformations of ${\mathbb C}^n$ generated by $G_f$ and $S$ is the semidirect product $G_f \rtimes S$ and the polynomial $f$ is $G_f \rtimes S$-invariant. Because of that, the group $G_f \rtimes S$ acts on the Milnor fibre $V_f=\big\{ \underline{x} \in {\mathbb C}^n \colon f(\underline{x})=1 \big\}$.

\begin{Remark}Elements of $G_f \rtimes S$ can be represented as pairs $(\underline{\lambda}, \sigma)$ with
$\underline{\lambda}=(\lambda_1, \dots , \lambda_n) \in G_f$, $\sigma \in S$. The multiplication in
$G_f \rtimes S$ is given by
\[
(\underline{\lambda}, \sigma) \cdot (\underline{\mu}, \tau) :=
(\underline{\lambda} \sigma(\underline{\mu}), \sigma \tau) ,
\]
where, for $\underline{\mu}=(\mu_1, \dots , \mu_n)$,
\[
\sigma(\underline{\mu}):= \big(\mu_{\sigma^{-1}(1)}, \dots , \mu_{\sigma^{-1}(n)}\big) .
\]
The action of the group $G_f \rtimes S$ on ${\mathbb C}^n$ is defined by
\[
(\underline{\lambda}, \sigma)\underline{x} = \big(\lambda_1x_{\sigma^{-1}(1)}, \dots , \lambda_nx_{\sigma^{-1}(n)}\big)
\]
$\big(\underline{x}=(x_1, \dots , x_n) \in {\mathbb C}^n\big)$.
\end{Remark}

The {\em Berglund--H\"ubsch} (BH) {\em transpose} of $f$ is
\[
\widetilde{f}(x_1, \dots, x_n)=\sum_{i=1}^n \prod_{j=1}^n x_j^{E_{ji}}
\]
(see~\cite{BH1}). One can show that the group $G_{\widetilde{f}}$ of diagonal symmetries of $\widetilde{f}$ is in a natural way isomorphic
to the group $G_f^*=\operatorname{Hom}(G_f,{\mathbb C}^*)$ of characters of $G_f$ (see, e.g., \cite[Proposition~2]{EG-BLMS}). Let $G$ be a subgroup of $G_f$. The {\em $($Berglund--Henningson$)$ dual subgroup} $\widetilde{G}$ in $G_{\widetilde{f}}$ is the set of characters $\alpha\colon G_f \to {\mathbb C}^*$ vanishing (i.e., being equal to~1) on the subgroup $G$ (\cite{BH2}, see also~\cite{Krawitz-thesis} or~\cite{Krawitz}). One has $|\widetilde{G}|=|G_f|/|G|$. The pair $\big(\widetilde{f},\widetilde{G}\big)$ is called the {\em Berglund--H\"ubsch--Henningson} (BHH) {\em dual}
of the pair $(f,G)$.

Let $S$ be a subgroup of $S_n$ preserving $f$ and let $G$ be a subgroup of $G_f$ invariant with respect to $S$, i.e., $\sigma(G)=G$ for any $\sigma \in S$. In this case, the semidirect product $G \rtimes S$ is defined and the polynomial $f$ is $G \rtimes S$-invariant. The BH-transpose $\widetilde{f}$ is also preserved by $S$ and the dual subgroup $\widetilde{G}$ is $S$-invariant. Therefore the group $\widetilde{G} \rtimes S$ preserves the polynomial $\widetilde{f}$. The pair $\big(\widetilde{f},\widetilde{G}\rtimes S\big)$ is called the {\em Berglund--H\"ubsch--Henningson--Takahashi} (BHHT) {\em dual} to the pair $(f,G\rtimes S)$ (see \cite{EG-IMRN}).

One says that the subgroup $S$ of $S_n$ satisfies the {\em parity condition} (PC) if, for any subgroup
$T \subset S$, one has
\[
\dim\big({\mathbb C}^n\big)^T \equiv n \quad \mbox{mod}\ 2,
\]
where $\big({\mathbb C}^n\big)^T:= \big\{ \underline{x} \in {\mathbb C}^n \colon \sigma \underline{x} =
\underline{x} \mbox{ for } \sigma \in T \big\}$ is the fixed point set of~$T$ (see~\cite{EG-IMRN}).

One can show that, if $S$ satisfies PC, then $S \subset A_n$. Moreover, if $S$ is a cyclic group (say, generated by~$s$), then $S$ satisfies PC if and only if $s \in A_n$.

\section[Orbifold Euler characteristic and fixed point sets of symmetries]{Orbifold Euler characteristic and fixed point sets\\ of symmetries} \label{sect:orbifold}

For a topological space $X$ with an action of a finite group $H$, its {\em orbifold Euler characteristic}
is defined by (see, e.g., \cite{AS,HH})
\begin{gather}\label{eqn:chi_orb}
{\chi}^{\rm orb}(X,H)=\frac{1}{\vert H\vert}\sum_{{(g,h)\in H^2:}\atop{gh=hg}}\chi\big(X^{\langle g,h\rangle}\big).
\end{gather}
Here $X^{\langle g,h\rangle}$ is the fixed point set of the subgroup of $H$ generated by~$g$ and~$h$, i.e.,
$X^{\langle g,h\rangle}=\{x\in X\colon gx=hx=x\}$, $\chi(\cdot)$ is the ``additive'' Euler characteristic defined as the alternating sum of the ranks of the cohomology groups with compact support. (One can show that ${\chi}^{\rm orb}(X,H)$ is an integer.) The {\em reduced orbifold Euler characteristic} is
\[
{\overline{\chi}}^{ \rm orb}(X,H)={\chi}^{\rm orb}(X,H)-{\chi}^{\rm orb}({\rm pt},H),
\]
where ${\rm pt}$ is the one point set with the unique action of~$H$. (If the group $H$ is abelian,
${\chi}^{\rm orb}({\rm pt},H)=\vert H\vert.$)

The orbifold Euler characteristic ${\chi}^{\rm orb}(X,H)$ is an additive invariant of $H$-spaces, i.e., spaces with an action of the group~$H$. The universal additive invariant of $H$-spaces is the equivariant Euler characteristic with values in the Burnside ring $A(H)$ of the group $H$ (see, e.g., \cite[Section~3]{RMS}). Therefore the orbifold Euler characteristic of $H$-spaces is the reduction of the equivariant one under a group homomorphism $A(H)\to{\mathbb Z}$. One can speculate that the symmetry property (coincidence up to sign) for the reduced orbifold Euler characteristics of the Milnor fibres of BHHT-dual pairs can be deduced from the non-abelian Saito duality. The results of \cite{EG-IMRN} imply that this is really the case if
\[
 \chi^{\rm orb}\big(G_f\rtimes S/H\rtimes T, G\rtimes S\big)=
 \chi^{\rm orb}\big(G_{\widetilde{f}}\rtimes S/\widetilde{H}\rtimes T,
 \widetilde{G}\rtimes S\big)
\]
for a subgroup $T$ of $S$ and for subgroups $H$ and $G$ of $G_f$ (with special properties).
Unfortunately, it is unclear how to prove this equation. In \cite{EG-1811} it was proved for $H=\{e\}$
(and thus $\widetilde{H}=G_{\widetilde{f}}$).

If $H$ is a subgroup of a finite group $K$, one has the {\em induction operation} $\operatorname{Ind}_H^K$ which converts $H$-spaces to $K$-spaces. For an $H$-space $X$, the space $\operatorname{Ind}_H^K X$ is the quotient of the Cartesian product $K\times X$ by the (right) action of the group $H$ defined by $(g,x)*h=\big(gh, h^{-1}x\big)$ ($g\in K$, $x\in X$, $h\in H$). The action of the group $K$ on $\operatorname{Ind}_H^K X$ is defined in the natural way: $g'*(g,x)=(g'g,x)$. One has the following important property of the orbifold Euler characteristic:
\[
{\chi}^{\rm orb}\big(\operatorname{Ind}_H^K X, K\big)={\chi}^{\rm orb}(X, H)
\]
(see \cite[Theorem~1]{GLM}).

The computation of the orbifold Euler characteristic ${\chi}^{\rm orb}(V_f, G \rtimes S)$ of the Milnor fibre of an invertible polynomial $f$ (in $n$ variables) with an action of a group $G\rtimes S$ ($G\subset G_f$, $S\subset S_n$) will be based on a decomposition of $V_f$ into its intersections with certain unions of the coordinate tori. For a subset $I\subset I_0=\{1, 2, \dots, n\}$, let
\[
{\mathbb C}^I:=\big\{\underline{x}=(x_1, \dots, x_n)\in {\mathbb C}^n\colon x_i=0 \text{ for }i\notin I\big\}
\]
and let
\[
({\mathbb C}^*)^I:=\big\{\underline{x}=(x_1, \dots, x_n)\in {\mathbb C}^n\colon x_i\neq 0 \text{ for }i\in I, x_i=0 \text{ for }i\notin I\big\} .
\]
One has
\[
{\mathbb C}^n=\coprod_{I\subset I_0} ({\mathbb C}^*)^I.
\]
Let $f^I$ be the restriction of the polynomial $f$ to ${\mathbb C}^I$, and let $V_f^I=V_f \cap ({\mathbb C}^*)^I$.
Each torus $({\mathbb C}^\ast)^I$ is invariant with respect to the action of the group $G_f$.
Let $G_f^I:=\big\{ \underline{\lambda} \in G_f \colon \underline{\lambda} \underline{x} =
\underline{x} \mbox{ for } \underline{x} \in ({\mathbb C}^\ast)^I \big\}$ be the isotropy group of the action of $G_f$ on $({\mathbb C}^\ast)^I$.

The group $S$ acts on the set $2^{I_0}$ of subsets of $I_0$. One can represent the space ${\mathbb C}^n$ as the disjoint unions
\[
{\mathbb C}^n=\coprod_{{\mathcal J}\in 2^{I_0}/S}\coprod_{J\in{\mathcal J}}({\mathbb C}^*)^J.
\]
The union of tori $\coprod_{J\in{\mathcal J}}({\mathbb C}^*)^J$ is invariant with respect to the action of the group $G \rtimes S$. Therefore
\[
{\chi}^{\rm orb}(V_f, G \rtimes S)=\sum_{{\mathcal J}\in 2^{I_0}/S}{\chi}^{\rm orb}\bigg(\coprod_{J\in{\mathcal J}}V_f^J, G \rtimes S\bigg) .
\]
For a subset $I\subset I_0$, let $S^I\subset S$ be the isotropy subgroup of $I$ for the $S$-action on $2^{I_0}$. Let $\overline{I} := I_0 \setminus I$ be the complement of $I$. One has $S^{\overline{I}}=S^I$.
One can see that, for a representative $I$ of an $S$-orbit ${\mathcal J}$, one has
\[
\coprod_{J\in{\mathcal J}}V_f^J=\operatorname{Ind}_{G\rtimes S^I}^{G\rtimes S} V_f^I
\]
(as a $G\rtimes S$-set). Therefore
\[
{\chi}^{\rm orb}\bigg(\coprod_{J\in{\mathcal J}}V_f^J,G \rtimes S\bigg)=
{\chi}^{\rm orb}\big(\operatorname{Ind}_{G\rtimes S^I}^{G\rtimes S} V_f^I, G\rtimes S\big)=
{\chi}^{\rm orb}\big(V_f^I, G\rtimes S^I\big).
\]

A polynomial $f$ is invertible if and only if it is the (Sebastiani--Thom) sum of ``atomic'' polynomials
in different (non-intersecting) sets of variables of one of the forms:
\begin{enumerate}\itemsep=0pt
 \item[1)] $x_1^{p_1}x_2+x_2^{p_2}x_3+\dots +x_{m-1}^{p_{m-1}}x_m+x_m^{p_m}$, $m\ge 1$ ({\em chain type});
 \item[2)] $x_1^{p_1}x_2+x_2^{p_2}x_3+\dots +x_{m-1}^{p_{m-1}}x_m+x_m^{p_m}x_1$, $m\ge 2$ ({\em loop type}).
\end{enumerate}
This classification appeared first in \cite{Kreuzer} with a reference to proofs in~\cite{KS}. Sometimes
(for example in~\cite{Krawitz-thesis}) one also distinguishes the so-called {\em Fermat type}:
$x_1^{p_1}$. Here we consider it as a~special case of the chain type with $m=1$. (There are some reasons to
consider it as a special case of the loop type with $m=1$ as well, writing it as $x_1^{p_1-1}x_1$.)

Let $f$ be an invertible polynomial and let $S$ be a permutation group preserving~$f$. An element~$\sigma$
of $S$ respects the decomposition of $f$ into atomic polynomials and sends each of them into an isomorphic one. For an atomic summand $f_{\alpha}$ of $f$, let~$N$ be the minimal power of~$\sigma$ which sends $f_{\alpha}$ to itself. One may have the following two (somewhat different) situations. First, the action of $\sigma^N$ on the set of variables of $f_{\alpha}$ may be trivial. This always happens if~$f_{\alpha}$ is of chain type. If $f_{\alpha}$ is of loop type, the action of $\sigma^N$ on the set of its variables may be non-trivial. A non-trivial automorphism of a loop can be a {\em rotation}. This means the following. The length~$m$ of the loop $f_{\alpha}=x_1^{p_1}x_2+\dots +x_m^{p_m}x_1$ is divisible by $\ell$, $0<\ell<m$, $m=k\ell$, the sequence $p_1, p_2, \dots, p_{k\ell}$ is $\ell$-periodic, that is, $p_{i+\ell}=p_i$, where the index $i$ is considered modulo $k\ell$, and the automorphism sends the variable $x_i$ to the variable $x_{i+s\ell}$ with $0<s<k$. Another option for a non-trivial automorphism is a {\em flip}. This means that there exists an index $q$ such that the automorphism sends the variable~$x_i$ to the variable $x_{q-i}$. Such an automorphism exists if and only if all the exponents $p_i$ are equal to~1. In this case, the polynomial $f_{\alpha}$ has either a~non-isolated critical point at the origin, or a non-degenerate
one (i.e., its Hessian is different from zero), depending on the parity of the length~$m$. We exclude flips from consideration, i.e., assume that~$\sigma^N$ is a~rotation.

For a computation of the orbifold Euler characteristic ${\chi}^{\rm orb}(V_f, G \rtimes S)$ with the use of
equation~(\ref{eqn:chi_orb}), one has to consider mutual fixed point sets $V_f^{\langle g,h\rangle}$ of pairs of
commuting elements $g, h\in G \rtimes S$. Here we shall consider the fixed point set
$(V_f^I)^{\langle g\rangle}$ ($I\subset I_0$) of an element $g=(\underline{\lambda}, \sigma)$, $\sigma \in S^I$,
and give a condition for it to be non-empty. First we consider the case $I=I_0$.

For an element $\sigma\in S$, let $\eta=(i_1,\dots, i_r)$ be a cycle in $\sigma$. For an element
$\underline{\lambda}\in G_f$, the {\em cycle product} of $\underline{\lambda}$ corresponding to $\eta$ is
$\lambda_{i_1}\cdots \lambda_{i_r}$. Let $\underline{x}=(x_1, x_2,\dots, x_n)\in ({\mathbb C}^*)^n$ be a fixed point of $g=(\underline{\lambda}, \sigma)\in G \rtimes S$. This means that
\begin{gather}\label{eqn:fixed_point}
(\underline{\lambda}, \sigma)\underline{x}=\underline{\lambda}\sigma(\underline{x})=\underline{x} ,
\end{gather}
where $\sigma(\underline{x})=(x_{\sigma^{-1}(1)}, x_{\sigma^{-1}(2)},\dots, x_{\sigma^{-1}(n)})$.
For a cycle $\eta=(i_1,\dots, i_r)$ in $\sigma$, equation~(\ref{eqn:fixed_point}) means that
\begin{gather}\label{eqn:cycle_fixed_point}
(\lambda_{i_1}x_{i_r}, \lambda_{i_2}x_{i_1},\dots, \lambda_{i_r}x_{i_{r-1}})=(x_{i_1}, x_{i_2},\dots, x_{i_r}) .
\end{gather}
Here $x_{i_j} \neq 0$, $j=1, \dots , r$, and therefore a solution of (\ref{eqn:cycle_fixed_point}) exists if and only if $\lambda_{i_1} \cdots \lambda_{i_r}=1$, i.e., if the corresponding cycle product is equal to~1. One can take, e.g.,
\[
(x_{i_1}, \dots , x_{i_r})=
(\lambda_{i_1}, \lambda_{i_1}\lambda_{i_2}, \dots , \lambda_{i_1} \cdots \lambda_{i_{r-1}},1).
\]
This implies that the fixed point set $\big(({\mathbb C}^*)^n\big)^{\langle (\underline{\lambda},\sigma) \rangle}$ is non-empty if and only if, for all cycles of~$\sigma$, the cycle products of $\underline{\lambda}$ are equal to~1.

For an element $\sigma \in S$, let $\ell$ be the number of cycles of the permutation $\sigma$.

\begin{Definition} The {\em cycle homomorphism} $C_\sigma$ is the map from $G_f$ to $({\mathbb C}^*)^\ell$ which sends an element $\underline{\lambda} \in G_f$ to the collection of the cycle products of $\underline{\lambda}$.
\end{Definition}

The discussion above means that the fixed point set $\big(({\mathbb C}^*)^n\big)^{\langle (\underline{\lambda},\sigma) \rangle}$
is non-empty if and only if $\underline{\lambda} \in \operatorname{Ker} C_\sigma$.
In this case one has $\dim \big({\mathbb C}^n\big)^{\langle (\underline{\lambda},\sigma) \rangle}=\dim \big({\mathbb C}^n\big)^{\langle \sigma \rangle}=\ell$.

\begin{Definition} For an element $\sigma \in S$, the {\em shift homomorphism} $A_\sigma$ is the map from $G_f$ to itself defined by
\[
A_\sigma(\underline{\lambda}) = \underline{\lambda}(\sigma(\underline{\lambda}))^{-1}.
\]
\end{Definition}

\begin{Remark} \label{rem:comm}
Two elements $(\underline{\lambda},\sigma)$ and $(\underline{\lambda}', \sigma')$ commute if and only if
$\sigma \sigma'=\sigma' \sigma$ and $\underline{\lambda}\sigma(\underline{\lambda}') =
\underline{\lambda}' \sigma'(\underline{\lambda})$. The latter condition is equivalent to
$A_{\sigma'}(\underline{\lambda})=A_\sigma(\underline{\lambda}')$.
\end{Remark}

\begin{Proposition} \label{prop:C=A}
One has
\[
\operatorname{Ker} C_\sigma = \operatorname{Im} A_\sigma.
\]
\end{Proposition}

\begin{proof}
It is easy to see that $\operatorname{Im} A_\sigma \subset \operatorname{Ker} C_\sigma$. Indeed, for
$\underline{\mu}= A_\sigma(\underline{\lambda})$ and for a cycle $(i_1, \dots , i_r)$ of $\sigma$, one has
$\mu_{i_j} = \lambda_{i_j}(\lambda_{i_{j-1}})^{-1}$, where the index $j-1$ is considered modulo $r$, and therefore the cycle product of $A_\sigma(\underline{\lambda})$ is equal to
\[
\mu_{i_1} \cdots \mu_{i_r} =
\lambda_{i_1}(\lambda_{i_r})^{-1} \lambda_{i_2}(\lambda_{i_1})^{-1} \cdots \lambda_{i_r}(\lambda_{i_{r-1}})^{-1} = 1.
\]
We shall show that the order $|\operatorname{Im} A_\sigma|$ of the subgroup $\operatorname{Im} A_\sigma$ is equal to the order
$|\operatorname{Ker} C_\sigma|$.

Let $f= \bigoplus_{\alpha \in A} f_\alpha$ be the representation of the invertible polynomial $f$ as the
Sebastiani--Thom sum of atomic polynomials $f_\alpha$. The permutation $\sigma$ sends each $f_\alpha$ to
an isomorphic one. Let us regard $f$ as $\bigoplus_{\omega \in A/\langle \sigma \rangle} \bigoplus_{\alpha \in \omega} f_\alpha$ where the first sum is over all orbits of the action of the group~$\langle \sigma \rangle$ (generated by~$\sigma$) on the set $A$ of indices. Since $G_f=\bigoplus_{\omega \in A/\langle \sigma \rangle} G_{\bigoplus_{\alpha \in \omega} f_\alpha}$ and~$\sigma$ preserves each summand $\bigoplus_{\alpha \in \omega} f_\alpha$, it is sufficient to prove the statement for
one block $\bigoplus_{\alpha \in \omega} f_\alpha$ with $\omega \in A/\langle \sigma \rangle$. Thus we may assume that $f= \bigoplus_{i=1}^N f_i$ where $f_i$ are isomorphic atomic polynomials and $\sigma$ sends $f_i$ to $f_{i+1}$ (the indices are considered modulo~$N$). The proof is somewhat different for the cases when the $f_i$ are of chain type and when the $f_i$ are of loop type.

1) Let
\[
f_i=x_{i,1}^{p_1}x_{i,2}+x_{i,2}^{p_2}x_{i,3}+\dots +x_{i,m-1}^{p_{m-1}}x_{i,m}+x_{i,m}^{p_m}
\]
be of chain
type, $i=1, \dots, N$. The permutation $\sigma$ sends the variable $x_{i,j}$ to the variable~$x_{i+1,j}$. The order $\vert G_f\vert$ of the group $G_f$ is equal to $\vert G_{f_1}\vert^N$ (in the case under consideration $\vert G_{f_1}\vert=p_1\cdots p_m$). Let ${\underline{\lambda}}_i=(\lambda_{i,1},\dots, \lambda_{i,m})$, $i=1, \dots, N$. For ${\underline{\lambda}}_i\in G_{f_i}$, one has
\begin{gather}\label{eqn:chain_power}
\lambda_{i,j}=\lambda_{i,1}^{(-1)^{j-1}p_1\cdots p_{j-1}}
\end{gather}
for $j=1,\dots, m$. The kernel $\operatorname{Ker}A_{\sigma}$ consists of the elements
$({\underline{\lambda}}_1, \dots, {\underline{\lambda}}_N)\in G_{f_1}^N$ such that
${\underline{\lambda}}_1={\underline{\lambda}}_2=\cdots= {\underline{\lambda}}_N$. Therefore
$\vert\operatorname{Ker}A_{\sigma}\vert=\vert G_{f_1}\vert$ and $\vert\operatorname{Im}A_{\sigma}\vert=\vert G_{f_1}\vert^{N-1}$.
Because of~(\ref{eqn:chain_power}) the cycle relation for $\lambda_{i,j}$ follows from the cycle relation for
$\lambda_{i,1}$. Therefore the kernel $\operatorname{Ker}C_{\sigma}$ consists of
$({\underline{\lambda}}_1, \dots, {\underline{\lambda}}_N)\in G_{f_1}^N$ such that
$\lambda_{1,1}\cdot\lambda_{2,1}\cdots\lambda_{N,1}=1$. In the elements of $G_f$, the components $\lambda_{1,1}$,
$\lambda_{2,1}$,\dots, $\lambda_{N,1}$ are arbitrary roots of degree $\vert G_{f_1}\vert$ of $1$. Therefore
$\vert\operatorname{Ker}C_{\sigma}\vert=\vert G_{f}\vert/\vert G_{f_1}\vert=\vert G_{f_1}\vert^{N-1}=
\vert\operatorname{Im} A_{\sigma}\vert$.

2) Let
\[
f_i=x_{i,1}^{p_1}x_{i,2}+x_{i,2}^{p_2}x_{i,3}+\dots +x_{i,m-1}^{p_{m-1}}x_{i,m}+x_{i,m}^{p_m}x_{i,1}
\]
be of
loop type, $i=1, \dots, N$.
The permutation $\sigma$ sends the variable $x_{i,j}$ to the variable $x_{i+1,j}$ for $1\le i\le N-1$ and sends
the variable $x_{N,j}$ to the variable $x_{1,j+L}$, where $0\le L\le m-1$ (the index $j$ is considered modulo $m$).
In this case $p_{j+L}=p_j$. If $L=0$, the proof literally coincides with the one for chains. Otherwise let
$\ell=\gcd{(L,m)}$. The sequence $p_1, \dots, p_m$ of the exponents is $\ell$-periodic, i.e., $p_{j+\ell}=p_j$.
Let $k:=\frac{m}{\ell}$, $P:=p_1\cdots p_{\ell}$. One has $\vert G_{f_i}\vert=P^k-(-1)^{k\ell}$,
$\vert G_{f}\vert=\vert G_{f_i}\vert^N$. Let ${\underline{\lambda}}_i=(\lambda_{i,1},\dots, \lambda_{i,m})$,
$i=1, \dots, N$. For ${\underline{\lambda}}_i\in G_{f_i}$, one has
$\lambda_{i,j}=\lambda_{i,1}^{(-1)^{j-1}p_1\cdots p_{j-1}}$ for $j=1,\dots, m$; cf.\ (\ref{eqn:chain_power}).
Because of this the cycle relation for $\lambda_{i,j}$ follows from the cycle relation for $\lambda_{i,1}$.
The kernel $\operatorname{Ker}A_{\sigma}$ consists of the elements
$({\underline{\lambda}}_1, \dots, {\underline{\lambda}}_N)\in G_{f_1}^N$ such that
\[
\lambda_{1,1}=\lambda_{2,1}=\dots=\lambda_{N,1}=\lambda_{1,\ell+1}=\dots = \lambda_{N,\ell+1}=\lambda_{1,2\ell+1}
=\dots=\lambda_{N,(k-1)\ell+1} ,
\]
i.e., $\underline{\lambda}_i=\underline{\lambda}_1$ for $1\le i\le N$ and, in addition,
\begin{gather}\label{eqn:loop_ker}
\lambda_{1,1}=\lambda_{1,\ell+1}=\lambda_{1,2\ell+1}=\dots=\lambda_{1,(k-1)\ell+1} .
\end{gather}
Since $\lambda_{1,\ell+1}=(\lambda_{1,1})^{(-1)^{\ell}P}$, equation~(\ref{eqn:loop_ker}) means that
$(\lambda_{1,1})^{(-1)^{\ell}P}=\lambda_{1,1}$, so $(\lambda_{1,1})^{P-(-1)^{\ell}}=1$, i.e., $\lambda_{1,1}$
is an arbitrary root of degree $P-(-1)^{\ell}$ of $1$. Therefore
\begin{gather*}%\label{eqn:ker_loop}
 \vert\operatorname{Ker}A_{\sigma}\vert=P-(-1)^{\ell}
\end{gather*}
and
\[
 \vert\operatorname{Im}A_{\sigma}\vert=\frac{P^k-(-1)^{k\ell}}{P-(-1)^{\ell}}.
\]
The cycle relation for $\lambda_{1,1}$ is
\[
\lambda_{1,1}\cdots\lambda_{N,1} \lambda_{1,\ell+1} \cdots \lambda_{N,\ell+1} \lambda_{1,\ell(k-1)+1} \cdots
\lambda_{N,\ell(k-1)+1} =1,
\]
i.e.,
\[
(\lambda_{1,1}\cdots\lambda_{N,1} )^{1+(-1)^\ell P + (-1)^{2 \ell} P^2 + \cdots + (-1)^{(k-1)\ell} P^{k-1}}=1.
\]
This means that the product $\lambda_{1,1}\cdots\lambda_{N,1}$ is an arbitrary root of degree
\[
P^{k-1}+(-1)^\ell P^{k-2} + \cdots + (-1)^{(k-1)\ell} = \frac{P^k -(-1)^{k\ell}}{P-(-1)^\ell}
\]
of 1.
Since $\lambda_{1,1}, \dots , \lambda_{N,1}$ are arbitrary roots of degree $P^k-(-1)^{k\ell}$ of 1, this implies
that
\[
|\operatorname{Ker} C_\sigma | = \frac{\big(P^k-(-1)^{k\ell}\big)^N}{P-(-1)^\ell} = |\operatorname{Im} A_\sigma |. \tag*{\qed}
\]
 \renewcommand{\qed}{}
\end{proof}

Let $I$ be a non-empty subset of $I_0=\{1, \dots , n\}$ and let $\sigma \in S^I$. The discussion above about
a~condition for $(\underline{\lambda}, \sigma)$ to have a non-empty fixed point set in $({\mathbb C}^*)^n$ gives the following analogue for~$({\mathbb C}^*)^I$: the fixed point set $\big(({\mathbb C}^*)^I\big)^{\langle (\underline{\lambda}, \sigma) \rangle}$ is non-empty if and only if, for all cycles of~$\sigma$ contained in $I$, the cycle products of $\underline{\lambda}$
are equal to~1. In this case, the dimension of $\big({\mathbb C}^I\big)^{\langle (\underline{\lambda},\sigma) \rangle}$ is equal to the dimension of $\big({\mathbb C}^I\big)^{\langle \sigma \rangle}$ and is equal to the number of cycles contained in~$I$.

Let the subset $I$ be such that the number of monomials of the polynomial~$f^I$ is equal to~$|I|$.
Let $\overline{I}=I_0 \setminus I$ and $|I|=k$. By renumbering the coordinates, we can assume without loss of generality that $I=\{1, \dots , k\}$. Then the matrix $E=(E_{ij})$ is of the form
\[
E = \left( \begin{matrix} E_I & 0 \\ \ast & E_{\overline{I}} \end{matrix} \right),
\]
where $E_I$ and $E_{\overline{I}}$ are square matrices of sizes $k \times k$ and $(n-k) \times (n-k)$ respectively, and $E_I$ is the matrix corresponding to $f^I$. Since $\det E \neq 0$, it follows that $\det E^I \neq 0$. Moreover, $f^I$ also has an isolated critical point at the origin, see, e.g., \cite[Proposition~5]{ET}. This implies that $f^I$ is an invertible polynomial.

\begin{Proposition} \label{prop:fixed_point} In the situation described above, the fixed point set $\big(({\mathbb C}^*)^I\big)^{\langle (\underline{\lambda}, \sigma) \rangle}$
is non-empty if and only if $\underline{\lambda} \in \operatorname{Ker} C_\sigma + G_f^I$.
\end{Proposition}

\begin{proof}Let $f=\bigoplus_{\omega \in A/\langle \sigma \rangle} \bigoplus_{\alpha \in \omega} f_\alpha$
be the decomposition of $f$ into the Sebastiani--Thom sum of polynomials $f_\alpha$ of atomic type.
One has $G_f=\bigoplus_{\omega \in A/\langle \sigma \rangle} G_{\bigoplus_{\alpha \in \omega} f_\alpha}$.
The subset $I$ is the disjoint union of the subsets $I_\omega$ where $I_\omega$
is the intersection of $I$ with the set of the indices corresponding to the coordinates in
$\bigoplus_{\alpha \in \omega} f_\alpha$. Since, for $\omega_1 \neq \omega_2$, the subsets $I_{\omega_1}$ and $I_{\omega_2}$ are disjoint, one has
$G_{\bigoplus_{\alpha \in \omega_1} f_\alpha} \subset G_f^{I_{\omega_2}}$. Therefore it is sufficient to prove
the statement for $f=\bigoplus_{i=1}^N f_i$ and $\sigma$ sending $f_i$ to $f_{i+1}$. If the polynomials
$f_i$, $i=1, \dots , N$, are of loop type, then $I$ consists of all the coordinates and the statement says
that $\underline{\lambda}$ has a non-empty fixed point set in the maximal torus if and only if
$\underline{\lambda} \in \operatorname{Ker} C_\sigma$, i.e., $\underline{\lambda}$ satisfies the cycle relation(s),
and $G_f^I=\{ 1 \}$.

If
\[ f_i= x_{i,1}^{p_1}x_{i,2} + x_{i,2}^{p_2}x_{i,3} + \cdots + x_{i,m-1}^{p_{m-1}}x_{i,m} +
x_{i,m}^{p_m}, \qquad i=1, \dots , N,
\]
are of chain type, then $I=I_r= \coprod_{i=1}^N \{ (i,r), \dots , (i,m) \}$ where $1 \leq r \leq m$. An element
$\underline{\lambda}=(\lambda_{i,j}) \in ({\mathbb C}^*)^{mN}$ belongs to $G_f$ if and only if $\lambda_{i,1}$ is a root
of degree $p_1 \cdots p_m$ of 1 and $\lambda_{i,j}=(\lambda_{i,1})^{(-1)^{j-1}p_1 \cdots p_{j-1}}$ for
$j=2, \dots , m$. In particular, $\lambda_{i,j}$ can be an arbitrary root of degree $p_j \cdots p_m$ of 1.
Since, for $j>r$, $\lambda_{i,j}=(\lambda_{i,r})^{(-1)^{j-r}p_r \cdots p_{j-1}}$, the cycle relation
$\lambda_{1,j} \lambda_{2,j} \cdots \lambda_{m,j}=1$ follows from the cycle relation
$\lambda_{1,r} \lambda_{2,r} \cdots \lambda_{m,r}=1$. In this case, one can write
$\lambda_{i,r}= \exp \frac{2 \pi \sqrt{-1} n_i}{p_r \cdots p_m}$ where $\sum_{i=1}^N n_i = 0$. Let
$\widetilde{\underline{\lambda}}$ be the element of $G_f$ defined by
$\widetilde{\lambda}_{i,1} = \exp \frac{2 \pi \sqrt{-1} n_i}{p_1 \cdots p_m}$. One has
$\widetilde{\underline{\lambda}} \in \operatorname{Ker} C_\sigma$
and $\underline{\lambda} \widetilde{\underline{\lambda}}^{-1} \in G_f^{I_r}$. This proves the statement.
\end{proof}

As above, let $I$ be a non-empty subset of $I_0$ such that the number of monomials of the polynomial $f^I$ is equal
to $|I|$ and let $g=(\underline{\lambda}, \sigma)$, $\sigma \in S^I$, be such that the fixed point set
$\big(({\mathbb C}^*)^I\big)^{\langle (\underline{\lambda}, \sigma) \rangle}$ is non-empty.

\begin{Proposition} \label{prop:Euler} One has
\begin{gather} \label{eq:Euler}
\chi\big(\big(V_f^I\big)^{\langle (\underline{\lambda}, \sigma) \rangle}\big) =
(-1)^{\dim \left({\mathbb C}^I\right)^{\langle \sigma \rangle} -1} \frac{|\operatorname{Ker} A_\sigma|}{\big|\operatorname{Ker} A_\sigma \cap G_f^I\big|}.
\end{gather}
\end{Proposition}

\begin{proof}The Euler characteristic under consideration is the Euler characteristic of the Milnor fibre of the restriction of the function $f$ to $\big({\mathbb C}^I\big)^{\langle (\underline{\lambda}, \sigma) \rangle}$. Let $f=\bigoplus_{\omega \in A/\langle \sigma \rangle} \bigoplus_{\alpha \in \omega} f_\alpha$ be the decomposition of $f$ into atomic polynomials and let $I_\alpha$ be the set of indices of the variables in~$f_\alpha$. The fixed point set $\big({\mathbb C}^I\big)^{\langle (\underline{\lambda}, \sigma) \rangle}$ is the direct sum of the spaces $\big( \bigoplus_{\alpha \in \omega} {\mathbb C}^{I_\alpha} \big)^{\langle (\underline{\lambda}, \sigma) \rangle}$ over all~$\omega$ such that $I \cap \big( \coprod_{\alpha \in \omega} I_\alpha \big)$ is non-empty. The restriction of $f$ to $\big({\mathbb C}^I\big)^{\langle (\underline{\lambda}, \sigma) \rangle}$ is the Sebastiani--Thom sum of its restrictions to $\big( \bigoplus_{\alpha \in \omega} {\mathbb C}^{I_\alpha}\big)^{\langle (\underline{\lambda}, \sigma) \rangle}$.
Therefore its Milnor fibre is homotopy equivalent to the join of the Milnor fibres of the restrictions of $f$ to $\big( \bigoplus_{\alpha \in \omega} {\mathbb C}^{I_\alpha}\big)^{\langle (\underline{\lambda}, \sigma) \rangle}$ and its Euler characteristic is equal up to sign to the product of the corresponding Euler characteristics for
$\bigoplus_{\alpha \in \omega} {\mathbb C}^{I_\alpha}$. The groups whose orders are in the numerator and in the denominator
of (\ref{eq:Euler}) are direct products of the corresponding groups for
$I \cap \big( \coprod_{\alpha \in \omega} I_\alpha\big)$. Therefore it is sufficient to prove (\ref{eq:Euler})
for the polynomial $\bigoplus_{\alpha \in \omega} f_\alpha$ with $\omega \in A/\langle \sigma \rangle$. Thus, as in Proposition~\ref{prop:fixed_point}, we may assume that $f=\bigoplus_{i=1}^N f_i$ ($f_i$ are atomic) and $\sigma$ sends $f_i$ to $f_{i+1}$. Again we have to distinguish between two cases.

1) Let
\[ f_i= x_{i,1}^{p_1}x_{i,2} + x_{i,2}^{p_2}x_{i,3} + \cdots + x_{i,m-1}^{p_{m-1}}x_{i,m} + x_{i,m}^{p_m}, \qquad i=1, \dots , N,
\]
be of chain type. The permutation $\sigma$ sends the variable $x_{i,j}$ to the variable $x_{i+1,j}$. The subset~$I$ (invariant with respect to $\sigma$) is of the form
\[
\coprod_{i=1}^N \{(i,r),(i,r+1), \dots , (i,m) \}
\]
with $1 \leq r \leq m$. The fixed point set $\big(({\mathbb C}^*)^I\big)^{\langle (\underline{\lambda}, \sigma) \rangle}$ consists of points of the form
\[
(\underline{\lambda}_1 \underline{y}, \underline{\lambda}_1\underline{\lambda}_2 \underline{y}, \dots ,
\underline{\lambda}_1 \cdots \underline{\lambda}_{m-1}\underline{y} , \underline{y}),
\]
where $\underline{\lambda}_i \in G_{f_i}$ (see the notations in the proof of Proposition~\ref{prop:fixed_point}),
\[
\underline{y}=(y_r, \dots , y_m)=(x_{N,r}, \dots , x_{N,m}).
\]
Therefore the restriction of $f$ to this set is equal to
\[
N\big(y_r^{p_r}y_{r+1} +y_{r+1}^{p_{r+1}}y_{r+2}+ \cdots + y_{m-1}^{p_{m-1}}y_m + y_m^{p_m}\big).
\]
The Euler characteristic of the intersection of its Milnor fibre with the corresponding torus is equal up to sign (not depending on $\underline{\lambda}$) to the determinant of the matrix of exponents (see, e.g.,
\cite[Theorem~7.1]{Varch}), which in our case is equal to $p_r \cdots p_m$. The group $\operatorname{Ker} A_\sigma$ consists of the elements of the form $(\underline{\lambda}_1, \underline{\lambda}_1, \dots , \underline{\lambda}_1)$, $\underline{\lambda}_1 \in G_{f_1}$, and its order $|\operatorname{Ker} A_\sigma|$ is equal to $p_1 \cdots p_m$. The group $\operatorname{Ker} A_\sigma \cap G_f^I$ consists of the elements of the form $(\underline{\lambda}_1, \underline{\lambda}_1, \dots , \underline{\lambda}_1)$ with $\underline{\lambda}_1 \in G_{f_1}^{I \cap I_1}$. This means that $\lambda_{1,1}$ is an arbitrary root of degree $p_1 \cdots p_{r-1}$ of 1 and therefore the order $\big|\operatorname{Ker} A_\sigma \cap G_f^I\big|$ is equal to $p_1 \cdots p_{r-1}$.

2) Let
\[
f_i= x_{i,1}^{p_1}x_{i,2} + x_{i,2}^{p_2}x_{i,3} + \cdots + x_{i,m-1}^{p_{m-1}}x_{i,m} + x_{i,m}^{p_m}x_{i,1}, \qquad i=1, \dots , N,
\]
be of loop type, $m=k\ell$, $p_i=p_{i+\ell}$, the permutation $\sigma$ sends the variable $x_{i,j}$ to the variable $x_{i+1,j}$ for $1 \leq i \leq N-1$ and sends the variable $x_{N,j}$ to the variable $x_{1,j+s\ell}$, $\gcd(s,k)=1$, and the set $I$ consists of the indices of all the variables. (One can say that in this case $I=I_0$.) In \cite[Proposition~3]{EG-IMRN}, it was shown that the element $(\underline{\lambda},\sigma) \in G \rtimes S$
(with non-empty fixed point set $\big(({\mathbb C}^*)^{Nk\ell}\big)^{\langle (\underline{\lambda}, \sigma) \rangle}$) is conjugate in $G_f \rtimes S$ to the element $(1, \sigma)$ (in fact by an element of the form $(\underline{\mu},1)$, $\underline{\mu} \in G_f$). This means that the fixed point set
$\big({\mathbb C}^I\big)^{\langle (\underline{\lambda}, \sigma) \rangle}$ is obtained from $\big({\mathbb C}^I\big)^{\langle \sigma \rangle}$ by the translation by $\underline{\mu}$. This translation preserves $f$. The fixed point set $\big(({\mathbb C}^*)^{Nk\ell}\big)^{\langle \sigma \rangle}$ consists of the points $\underline{x}=(x_{i,j})$ with $x_{i,j}=x_{1,j}$ for $i=1, \dots , N$, $j=1, \dots , k\ell$ and $x_{1,j+\ell}=x_{1,j}$ (the index $j$ is considered modulo $k\ell$). Therefore, as coordinates on $\big(({\mathbb C}^*)^{Nk\ell}\big)^{\langle \sigma \rangle}$, one can take $y_j=x_{1,j}$ for $1 \leq j \leq \ell$ and the restriction of the polynomial $f$ to this subspace is equal to
\begin{gather} \label{eq:loop_fix}
kN \big(y_1^{p_1}y_2+y_2^{p_2}y_3 + \cdots + y_{\ell-1}^{p_{\ell-1}}y_\ell + y_\ell^{p_\ell} y_1\big).
\end{gather}
As in 1), the Euler characteristic of the intersection of its Milnor fibre with the maximal torus is equal up to sign (not depending on $\underline{\lambda}$) to $P-(-1)^\ell$ where $P=p_1 \cdots p_\ell$. We have
$|\operatorname{Ker} A_\sigma|=P-(-1)^\ell$ and $G_f^I=\{ 1 \}$ and therefore $\big|\operatorname{Ker} A_\sigma \cap G_f^I\big|=1$. Note that, for $\ell=1$, the polynomial~(\ref{eq:loop_fix}) is equal to $kNy_1^{p_1}y_1$ (i.e., is of Fermat type) and the equation for the Euler characteristic holds as well. This proves the statement up to sign. However the sign of the Euler characteristic of the Milnor fibre is determined by the dimension.
\end{proof}

\section{Symmetry for cyclic permutation groups} \label{sect:Symmetry}
Let $f$ be an invertible polynomial (in $n$ variables) and let $S\subset S_n$ be a subgroup of the group of
permutations of the coordinates preserving~$f$. Let $G$ be an $S$-invariant subgroup of $G_f$, and let the pair $\big(\widetilde{f}, \widetilde{G} \rtimes S\big)$ be the BHHT-dual to $(f,G \rtimes S)$.

\begin{Theorem}\label{theo:Main} If $S$ is a cyclic group satisfying the condition PC, then
 \begin{gather}\label{eqn:Main}
 {\overline{\chi}}^{ \rm orb}(V_f,G \rtimes S)=(-1)^n {\overline{\chi}}^{ \rm orb}\big(V_{\widetilde{f}},\widetilde{G} \rtimes S\big) .
 \end{gather}
\end{Theorem}

One has ${\mathbb C}^n=\coprod_{I\subset I_0}({\mathbb C}^*)^I$ ($I_0=\{1,\dots, n\}$) and therefore
\[V_f=\coprod_{I\subset I_0,\, I\ne\varnothing} V_f^I
\]
($0\notin V_f$). The group $S$ acts on
$2^{I_0}$. One has the decomposition
\[
{\mathbb C}^n \setminus \{ 0 \} =\coprod_{{\mathcal J}\in\left(2^{I_0}\setminus\{\varnothing\}\right)/S}\coprod_{J\subset {\mathcal J}}({\mathbb C}^*)^J,
\]
where the (disjoint)
unions are over the orbits ${\mathcal J}$ of the $S$-action except the one of the empty set and over the elements of
the orbit. Therefore
\[
V_f=\coprod_{{\mathcal J}\in\left(2^{I_0}\setminus\{\varnothing\}\right)/S}\coprod_{J\subset {\mathcal J}}V_f^J.
\]

For $I\subset I_0$, let $S^I$ be the isotropy subgroup of $I$ for the action of $S$ on $2^{I_0}$. It is easy to see that
\[
\coprod_{J\subset {\mathcal J}}V_f^J=\operatorname{Ind}_{G\rtimes S^I}^{G\rtimes S} V_f^I
\]
for an element $I$ of ${\mathcal J}$. One has
\[
{\chi}^{\rm orb}\big(\operatorname{Ind}_{G\rtimes S^I}^{G\rtimes S} V_f^I, G\rtimes S\big)=
{\chi}^{\rm orb}\big(V_f^I, G\rtimes S^I\big)
\]
(see~\cite[Theorem~1]{GLM}).

Recall that, for a subset $I\subset I_0$, $f^I$ denotes the restriction of $f$ to ${\mathbb C}^I$. The polynomial $f^I$ has
not more than $\vert I\vert$ monomials and it has $\vert I\vert$ monomials if and only if
$\widetilde{f}^{\overline{I}}$ has $\vert \overline{I}\vert$ monomials. If $f^I$ has less than $\vert I\vert$
monomials, then
\[
{\chi}^{\rm orb}\big(V_f^I, G\rtimes S^I\big)=0 .
\]
This follows from~\cite[Lemma~1]{EG-IMRN}, which says that, in this case, the equivariant Euler characteristic $\chi^{G_f\rtimes S^I}\big(V_f^I\big)$ with values in the Burnside ring $A(G_f\rtimes S^I)$ is equal to zero, together with the facts that ${\chi}^{\rm orb}\big(V_f^I, G\rtimes S^I\big)$ is a reduction of $\chi^{G\rtimes S^I}\big(V_f^I\big)$ and $\chi^{G\rtimes S^I}\big(V_f^I\big)$ is a reduction of $\chi^{G_f\rtimes S^I}\big(V_f^I\big)$. Therefore, if $f^I$ has less than $\vert I\vert$ monomials, then both ${\chi}^{\rm orb}\big(V_f^I, G\rtimes S^I\big)$ and ${\chi}^{\rm orb}\big(V_{\widetilde{f}}^{\overline{I}}, \widetilde{G}\rtimes S^I\big)$ are equal to zero.

One has
\begin{align*}
{\chi}^{\rm orb}\big(V_f^I, G\rtimes S^I\big) & =\frac{1}{\vert G\vert\cdot\vert S^I\vert}
 \sum_{{((\underline{\lambda},\sigma),(\underline{\lambda}',\sigma'))\in (G\rtimes S^I)^2:}\atop
 {(\underline{\lambda},\sigma)(\underline{\lambda}',\sigma')=
 (\underline{\lambda}',\sigma')(\underline{\lambda},\sigma)}}
 \chi\big(\big(V_f^I\big)^{\langle(\underline{\lambda},\sigma),(\underline{\lambda}',\sigma')\rangle}\big)\\
 &= \frac{1}{\vert S^I\vert}\sum_{{(\sigma,\sigma')\in (S^I)^2\colon}\atop {\sigma\sigma'=\sigma'\sigma}}
 \frac{1}{\vert G\vert}
 \sum_{{(\underline{\lambda},\underline{\lambda}')\in G^2\colon }\atop
 {(\underline{\lambda},\sigma)(\underline{\lambda}',\sigma')=
 (\underline{\lambda}',\sigma')(\underline{\lambda},\sigma)}}
 \chi\big(\big(V_f^I\big)^{\langle(\underline{\lambda},\sigma),(\underline{\lambda}',\sigma')\rangle}\big).
\end{align*}
Let
\begin{gather}\label{eqn:chi_sigma}
 \chi_{f,G}^I(\sigma,\sigma'):=\frac{1}{\vert G\vert}
 \sum_{{(\underline{\lambda},\underline{\lambda}')\in G^2:}\atop
 {(\underline{\lambda},\sigma)(\underline{\lambda}',\sigma')=
 (\underline{\lambda}',\sigma')(\underline{\lambda},\sigma)}}
 \chi\big(\big(V_f^I\big)^{\langle(\underline{\lambda},\sigma),(\underline{\lambda}',\sigma')\rangle}\big).
\end{gather}

One has
\begin{gather*}
 {\chi}^{\rm orb}\big({\rm pt}, G\rtimes S^I\big) = \frac{1}{\vert S^I\vert}\sum_{(\sigma,\sigma')\in (S^I)^2}
\frac{1}{\vert G\vert}\big\vert\{(\underline{\lambda},\underline{\lambda}')\in G^2\colon
(\underline{\lambda},\sigma)(\underline{\lambda}',\sigma') =(\underline{\lambda}',\sigma')(\underline{\lambda},\sigma)
\}\big\vert.
\end{gather*}
Let
\[
 \chi_{f,G}^{\varnothing}(\sigma,\sigma'):= - \frac{1}{\vert G\vert}\big\vert\{(\underline{\lambda},\underline{\lambda}')
 \in G^2\colon
(\underline{\lambda},\sigma)(\underline{\lambda}',\sigma') =(\underline{\lambda}',\sigma')(\underline{\lambda},\sigma)
\}\big\vert.
\]

In these terms, one has
\begin{gather} \label{eqn:chi_f,G}
{\overline{\chi}}^{ \rm orb}\big(V_f, G\rtimes S^I\big) =
\sum_{{\mathcal J}=[I]\in 2^{I_0}/S} \frac{1}{\vert S^I\vert}
\sum_{{(\sigma,\sigma')\in (S^I)^2\colon }\atop {\sigma\sigma'=\sigma'\sigma}} \chi^I_{f,G}(\sigma, \sigma'),
\end{gather}
where the first sum runs over all the $S$-orbits in $2^{I_0}$ including the orbit of the empty set,
$[I]$~denotes the orbit of the subset $I$.

Let $S^I$ be a cyclic group ($\cong {\mathbb Z}_q$) and let $s$ be a generator of $S^I$.

\begin{Proposition} \label{prop:red}
 Let $\sigma=s^m$, $\sigma'=s^{m'}$. For $m^*=\gcd(m, m', q)$, one has
 $\chi_{f,G}^I(\sigma, \sigma')=\chi_{f,G}^I(s^{m^*}, 1)$.
\end{Proposition}

\begin{proof}
 For $(\underline{\lambda},\sigma)$ and $(\underline{\lambda}',\sigma')$ from $G \rtimes S$, let
 $\big(\underline{\nu},\sigma(\sigma')^{-1}\big):=(\underline{\lambda},\sigma)(\underline{\lambda}',\sigma')^{-1}$,
 where $\underline{\nu}=\underline{\nu}(\underline{\lambda},\underline{\lambda}')$. Then
 $(\underline{\lambda},\sigma)$ commutes with $(\underline{\lambda}',\sigma')$ if and only if
 $(\underline{\lambda},\sigma)$ commutes with $(\underline{\nu},\sigma(\sigma')^{-1})$.
 Moreover one has
 $\big(V_f^I\big)^{\langle(\underline{\lambda},\sigma),(\underline{\lambda}',\sigma')\rangle}=
 \big(V_f^I\big)^{\langle(\underline{\lambda},\sigma),(\underline{\nu},\sigma(\sigma')^{-1})\rangle}$.
 Therefore the correspondence
 $\big((\underline{\lambda},\sigma),(\underline{\lambda}',\sigma')\big)\longleftrightarrow
 \big((\underline{\lambda},\sigma),
 \big(\underline{\nu}(\underline{\lambda},\underline{\lambda}'),\sigma(\sigma')^{-1}\big)\big)
 $
 preserves the summands in~(\ref{eqn:chi_sigma}) and therefore
 \[
 \chi_{f,G}^I(\sigma,\sigma')=\chi_{f,G}^I\big(\sigma,\sigma(\sigma')^{-1}\big).
 \]
Then the Euclidian algorithm implies the statement. The arguments are valid for $I=\varnothing$ as well.
\end{proof}

Now we shall compute $\chi_{f,G}^I(\sigma, 1)$ for $I$ such that $f^I$ has $\vert I\vert$ monomials.
(This includes $I=\varnothing$.)

\begin{Proposition} \label{prop:chisigma1}
 In this case,
 \begin{gather}
 \chi_{f,G}^I(\sigma, 1)= (-1)^{d^I_\sigma-1}\frac{1}{\vert G\vert}\cdot
 \frac{\vert \operatorname{Ker}A_{\sigma}\vert}{\vert \operatorname{Ker}A_{\sigma}\cap G^I_f\vert}\nonumber\\
 \hphantom{\chi_{f,G}^I(\sigma, 1)=}{}\times \big\vert G\cap\big(\operatorname{Im}A_{\sigma}+G_f^I\big)\big\vert\cdot\big\vert G\cap\big(\operatorname{Ker}A_{\sigma}\cap G_f^I\big)\big\vert,\label{eqn:chisigma1}
 \end{gather}
 where $d^I_\sigma= \dim\big({\mathbb C}^I\big)^{\langle \sigma \rangle}$.
\end{Proposition}

\begin{proof}
 First, let $I$ be non-empty. The fixed
 point set of $(\underline{\lambda},\sigma)$ in $({\mathbb C}^*)^I$ is not empty if and only if
 $\underline{\lambda}\in \operatorname{Ker}C_{\sigma}+G_f^I$ (Proposition~\ref{prop:fixed_point}).
 Thus the number of the elements $\underline{\lambda}\in G$ with a non-empty fixed point set
 $\big(({\mathbb C}^*)^I\big)^{\langle(\underline{\lambda},\sigma)\rangle}$ is equal to
 $\big\vert G\cap\big(\operatorname{Ker}C_{\sigma}+G_f^I\big)\big\vert=\big\vert G\cap\big(\operatorname{Im}A_{\sigma}+G_f^I\big)\big\vert$
 (see Proposition~\ref{prop:C=A}).

 The fixed point set of $(\underline{\lambda}',1)$ in $({\mathbb C}^*)^I$ is empty for
 $\underline{\lambda}'\notin G_f^I$ and is equal to $({\mathbb C}^*)^I$ for $\underline{\lambda}'\in G_f^I$.
 The element $(\underline{\lambda}',1)\in G\rtimes S$ commutes with $(\underline{\lambda},\sigma)$ if and only if
 $A_{\sigma}(\underline{\lambda}')=A_{1}(\underline{\lambda})=1$, i.e., if
 $\underline{\lambda}'\in\operatorname{Ker}A_{\sigma}$. Thus, for a fixed $\underline{\lambda}\in G$ with a non-empty
 fixed point set $\left(({\mathbb C}^*)^I\right)^{\langle (\underline{\lambda},\sigma) \rangle}$, the number of elements
 $(\underline{\lambda}',1)\in G\rtimes S$ commuting with $(\underline{\lambda},\sigma)$ and having a non-empty
 fixed point set in $({\mathbb C}^*)^I$ (coinciding with $({\mathbb C}^*)^I$) is equal to
 $\big\vert G\cap\big(\operatorname{Ker}A_{\sigma}\cap G_f^I\big)\big\vert$. In this case
 $\big(V_f^I\big)^{\langle(\underline{\lambda},\sigma),(\underline{\lambda}',1)\rangle}=
 \big(V_f^I\big)^{\langle(\underline{\lambda},\sigma)\rangle}$ (since $\underline{\lambda}'\in G_f^I$) and
 according to Proposition~\ref{prop:Euler} one has
 \[
 \chi\big(\big(V_f^I\big)^{\langle(\underline{\lambda},\sigma),(\underline{\lambda}',1)\rangle}\big) =(-1)^{d^I_\sigma-1}
 \frac{\vert \operatorname{Ker}A_{\sigma}\vert}{\big\vert \operatorname{Ker}A_{\sigma}\cap G^I_f\big\vert}.
 \]
 This proves the statement for a non-empty $I$.

 Let us show that equation~(\ref{eqn:chisigma1}) holds for $I=\varnothing$ as well. In this case $G_f^I=G_f$
 and the right hand side of~(\ref{eqn:chisigma1}) degenerates to $-\vert G\cap \operatorname{Ker}A_{\sigma}\vert$
 (since $\operatorname{Ker}A_{\sigma}\cap G^I_f=\operatorname{Ker}A_{\sigma}$,
 $G\cap\big(\operatorname{Im}A_{\sigma}+G_f^I\big)=G$, $G\cap\big(\operatorname{Ker}A_{\sigma}\cap G_f^I\big)=G\cap\operatorname{Ker}A_{\sigma}$).
 For an arbitrary element $\underline{\lambda}\in G$ (their number being equal to $\vert G\vert$),
 $(\underline{\lambda}, \sigma)$ commutes with $(\underline{\lambda}', 1)$ if and only if
 $A_{\sigma}(\underline{\lambda}')=A_{1}(\underline{\lambda})=1$ (see Remark~\ref{rem:comm}), i.e., if
 $\underline{\lambda}'\in G\cap \operatorname{Ker}A_{\sigma}$. Therefore
 \[
 \big\vert\big\{(\underline{\lambda},\underline{\lambda}')\in G^2\colon (\underline{\lambda}, \sigma) \text{ and }
 (\underline{\lambda}', 1) \text{ commute}\big\}\big\vert=\vert G\vert\cdot\vert G\cap \operatorname{Ker}A_{\sigma}\vert
 \]
 and $\chi_{f,G}^I(\sigma, 1)=-\vert G\cap \operatorname{Ker}A_{\sigma}\vert$.
\end{proof}

\begin{proof}[Proof of Theorem~\ref{theo:Main}] From (\ref{eqn:chi_f,G}) together with Proposition~\ref{prop:red}, it follows that it is sufficient to show that, for $I$ such that $f^I$ has $\vert I\vert$ monomials, one has
\begin{gather}\label{eqn:equal_chi_sign}
 \chi_{f,G}^I(\sigma,1)= (-1)^n
 \chi_{\widetilde{f},\widetilde{G}}^{\overline{I}}(\sigma,1) ,
\end{gather}
$\sigma \in S^I$.
According to Proposition~\ref{prop:Euler}, the signs in all non-zero summands on the left hand side
and on the right hand side of (\ref{eqn:equal_chi_sign}) (see Definition~(\ref{eqn:chi_sigma})) are
\[
(-1)^{\dim \left({\mathbb C}^I\right)^{\langle \sigma \rangle}-1} \qquad \text{and} \qquad
 (-1)^n (-1)^{\dim\big({\mathbb C}^{\overline{I}}\big)^{\langle\sigma \rangle}-1},
 \] respectively. The condition PC gives
\[
\dim\big({\mathbb C}^n\big)^{\langle \sigma \rangle}=\dim\big({\mathbb C}^I\big)^{\langle \sigma \rangle}+
\dim\big({\mathbb C}^{\overline{I}}\big)^{\langle \sigma \rangle}\equiv n \mod 2 .
\]
Therefore the signs on the left hand side and on the right hand side coincide and to prove the statement it is
sufficient to show that in this case
\[
 \big\vert \chi_{f,G}^I(\sigma,1)\big\vert= \big\vert\chi_{\widetilde{f},\widetilde{G}}^{\overline{I}}(\sigma,1)\big\vert .
\]

One has
\begin{gather*}
\big\vert G\cap\big(\operatorname{Im}A_{\sigma} + G_f^I\big)\big\vert =
\frac{\vert G\vert\cdot\big\vert\operatorname{Im}A_{\sigma} + G_f^I\big\vert}{\big\vert G+\big(\operatorname{Im}A_{\sigma} + G_f^I\big)\big\vert}, \\
\big\vert G\cap\big(\operatorname{Ker}A_{\sigma}\cap G_f^I\big)\big\vert =
\frac{\vert G\vert\cdot\big\vert\operatorname{Ker}A_{\sigma}\cap G_f^I\big\vert}{\big\vert G+\big(\operatorname{Ker}A_{\sigma}\cap G_f^I\big)\big\vert}.
\end{gather*}
Therefore Proposition~\ref{prop:chisigma1} gives
\[
 \big\vert \chi_{f,G}^I(\sigma,1)\big\vert=\vert\operatorname{Ker}A_{\sigma}\vert\cdot\vert G\vert\cdot
\frac{\big\vert\operatorname{Im}A_{\sigma}+ G_f^I\big\vert}
{\big\vert G+\big(\operatorname{Im}A_{\sigma}+ G_f^I\big)\big\vert\cdot\big\vert G+\big(\operatorname{Ker}A_{\sigma}\cap G_f^I\big)\big\vert}.
\]
The subgroup of $G_{\widetilde{f}}$ dual to $G_f^I$ is $G_{\widetilde{f}}^{\overline{I}}$
(see~\cite[Lemma~1]{EG-BLMS}).
The subgroups $\operatorname{Ker}A_{\sigma}$ and $\operatorname{Im}A_{\sigma}$ of $G_f$ are dual to the subgroups
$\operatorname{Im}A^*_{\sigma}$ and $\operatorname{Ker}A^*_{\sigma}$ of $G_{\widetilde{f}}$, respectively.
(The homomorphism $A^*_{\sigma}\colon G_{\widetilde{f}}\to G_{\widetilde{f}}$ is in fact the corresponding homomorphism
$A_{\sigma}$ on this group. We keep the notation $A^*_{\sigma}$ to recall what it is acting on.) In particular,
$\vert \operatorname{Ker}A_{\sigma}\vert=\frac{\vert G_f\vert}{\vert \operatorname{Im}A^*_{\sigma}\vert}$.
Therefore the subgroups dual to $\operatorname{Im}A_{\sigma}+ G_f^I$, $G+\big(\operatorname{Im}A_{\sigma}+ G_f^I\big)$, and
$G+\big(\operatorname{Ker}A_{\sigma}\cap G_f^I\big)$ are $\operatorname{Ker}A^*_{\sigma}\cap G_{\widetilde{f}}^{\overline{I}}$,
$\widetilde{G}\cap\big(\operatorname{Ker}A^*_{\sigma} \cap G_{\widetilde{f}}^{\overline{I}}\big)$, and
$\widetilde{G} \cap\big(\operatorname{Im}A^*_{\sigma} + G_{\widetilde{f}}^{\overline{I}}\big)$, respectively. Hence one gets
\begin{gather}\label{eqn:chi_sigma_final}
 \big\vert \chi_{f,G}^I(\sigma,1) \big\vert =
 \frac{\vert G \vert}{\vert G_f \vert} \cdot \frac{\vert\operatorname{Ker}
 A_{\sigma}\vert}{\big\vert\operatorname{Ker}A^*_{\sigma} \cap G_{\widetilde{f}}^{\overline{I}}\big\vert}\cdot
 \big\vert\widetilde{G}\cap\big(\operatorname{Ker}A^*_{\sigma} \cap G_{\widetilde{f}}^{\overline{I}}\big)\big\vert
 \cdot
 \big\vert\widetilde{G}\cap\big(\operatorname{Im}A^*_{\sigma}+ G_{\widetilde{f}}^{\overline{I}}\big)\big\vert.
\end{gather}
One has $\vert\operatorname{Ker}A_{\sigma}\vert=\frac{\vert G_f\vert}{\vert\operatorname{Im}A_{\sigma}\vert}$. Since the subgroup
dual to $\operatorname{Im}A_{\sigma}$ is $\operatorname{Ker}A^*_{\sigma}$, one gets
$\vert\operatorname{Ker}A_{\sigma}\vert=\vert\operatorname{Ker}A^*_{\sigma}\vert$. Therefore the right hand side of~(\ref{eqn:chi_sigma_final}) coincides with $\big\vert \chi_{\widetilde{f},\widetilde{G}}^{\overline{I}}(\sigma,1) \big\vert$.

This proves Theorem~\ref{theo:Main}.
\end{proof}

\begin{Remark}The result in \cite[Theorem~4]{EG-1811} is a particular case of Theorem~\ref{theo:Main}.
\end{Remark}

\begin{Remark} One can show that
 \[
 {\overline{\chi}}^{ \rm orb}(V_f,G \rtimes S)=-{\chi}^{\rm orb}\big({\mathbb C}^n,V_f;G \rtimes S\big),
 \]
 where ${\chi}^{\rm orb}\big({\mathbb C}^n,V_f;G \rtimes S\big)$ is the orbifold Euler characteristic of the pair $({\mathbb C}^n,V_f)$
 (which is equal to ${\chi}^{\rm orb}\big({\mathbb C}^n/V_f,G \rtimes S\big)$). Therefore equation~(\ref{eqn:Main}) is equivalent to
 \[
 {\chi}^{\rm orb}\big({\mathbb C}^n,V_f;G \rtimes S\big)=(-1)^n{\chi}^{\rm orb}\big({\mathbb C}^n,V_{\widetilde{f}};\widetilde{G} \rtimes S\big).
 \]
\end{Remark}

\section{Examples} \label{sect:Examples}

Theorem~\ref{theo:Main} gives the orbifold Euler characteristics of the Milnor fibres of pairs BHHT-dual to
a number of examples from \cite[Table~2]{Mukai} which are not present in the table. For instance, in Examples~33 and~34 from \cite{Mukai} one has the pairs $(f,G \rtimes S)$, where
\[
f=x_1^4x_2+x_2^4x_1+x_3^4x_4+x_4^4x_3+x_5^5, \qquad G=\left\langle\tfrac{1}{3}(1,2,0,0,0), \tfrac{1}{3}(0,0,1,2,0), J\right\rangle,
\]
$S=\langle(13)(24)\rangle$ in Example~33 and $S=\langle(12)(34)\rangle$ in Example~34. Here $\frac{1}{m}(a_1,\dots, a_5)$ means the operator $\operatorname{diag} \big(\exp\frac{2\pi a_1 {\rm i}}{m},\dots,\exp\frac{2\pi a_5{\rm i}}{m}\big)$ and $J=\frac{1}{5}(1,1,1,1,1)$ is the exponential grading operator. For the BHHT-dual pairs $\big(\widetilde{f}, \widetilde{G} \rtimes S\big)$ one has $\widetilde{f}=f$ and $\widetilde{G}=\big\langle\frac{1}{5}(1,1,4,4,0), J\big\rangle$. The data from \cite[Table~2]{Mukai} (compiled with the use of a computer) together with Theorem~\ref{theo:Main} give that in the cases dual to Examples~33 and~34 the orbifold Euler characteristics of the Milnor fibres are equal to $8$ and $-16$, respectively.

In the same way, one gets the following table of orbifold Euler characteristics of the Milnor fibres of the
BHHT-dual pairs for some other examples. Here the first line indicates the number of the example from
\cite[Table~2]{Mukai} and the second one gives the orbifold Euler characteristic for the dual pair.
(Let us recall that in all these cases the dual pairs are not present in \cite[Table~2]{Mukai}.)

\begin{center}
\begin{tabular}{cccccccccccccccc}
\hline
 36 & 45 & 46 & 48 & 50 & 53 & 60 & 65 & 68 & 69 & 71 & 72 & 77 & 79 & 81 & 87\\
 8 & 40 & $-16$ & 16 & 8 & 16 & $-16$ & 8 & 112 & 112 & 112 & 88 & 112 & 88 & 88 & 40\\
\hline
\end{tabular}
\end{center}

\subsection*{Acknowledgements}
This work was partially supported by DFG. The work of the second author (Sections~\ref{sect:invertible_poly} and \ref{sect:Symmetry}) was supported by the grant 16-11-10018 of the Russian Foundation for Basic Research. We are very grateful to the referees of the paper for their useful comments.

\pdfbookmark[1]{References}{ref}
\LastPageEnding

\end{document}